\documentclass[a4paper]{amsart}
\newtheorem{thm}{Theorem}[section]

\newtheorem{lema}[thm]{Lemma}
\newtheorem{prop}[thm]{Proposition}

\theoremstyle{definition}

\theoremstyle{remark}
\newtheorem{rem}[thm]{Remark}

\numberwithin{equation}{section}

\newcommand{\R}{\mathbb R}
\newcommand{\N}{\mathbb N}

\newcommand{\lam}{\lambda}

\begin{document}

\title[Precise asymptotic of eigenvalues]{Precise asymptotic of eigenvalues of resonant quasilinear systems}

\author[J. Fern\'andez Bonder and J. P. Pinasco]{Juli\'an Fern\'andez Bonder and Juan P. Pinasco}

\address{Departamento de Matem\'atica, FCEyN - Universidad de Buenos Aires\hfill\break \indent Ciudad Universitaria, Pabell\'on I  (1428) Buenos Aires,
Argentina.}

\email[J. Fern\'andez Bonder]{jfbonder@dm.uba.ar}

\urladdr[J. Fern\'andez Bonder]{http://mate.dm.uba.ar/~jfbonder}

\email[J.P. Pinasco]{jpinasco@dm.uba.ar}

\urladdr[J.P. Pinasco]{http://mate.dm.uba.ar/~jpinasco}



% ----------------------------------------------------------------
\begin{abstract}
In this work we study the sequence of variational eigenvalues of a system of resonant type involving $p-$ and $q-$laplacians on $\Omega \subset \R^N$, with a coupling term depending on two parameters $\alpha$ and $\beta$ satisfying $\alpha/p + \beta/q = 1$. We show that the order of growth of the $k^{th}$ eigenvalue depends on $\alpha+\beta$, $\lam_k = O(k^{\frac{\alpha+\beta}{N}})$.
\end{abstract}

\maketitle
% ----------------------------------------------------------------

\section{Introduction}

This paper is devoted to the study of the asymptotic behavior of eigenvalues of resonant quasilinear systems
\begin{equation}\label{uno}
\begin{cases}
-\Delta_p u =  \lambda \alpha |u|^{\alpha -2}u|v|^{\beta}\\
-\Delta_q v =  \lambda \beta |u|^{\alpha}|v|^{\beta-2}v,
\end{cases}\quad \mbox{in } \Omega
\end{equation}
with Dirichlet boundary condition
\begin{equation}\label{dos}
u(x) =v(x) =0 \qquad \qquad \mbox{on } \partial \Omega.
\end{equation}
Here, $\Omega \subset \R^N$ is a bounded domain with smooth boundary $\partial\Omega$, the $s-$Laplace operator is $\Delta_s u = \mbox{div}(|\nabla u|^{s-2} \nabla u)$, the exponents satisfy $1 < p,q < +\infty$, and the positive parameters $\alpha, \beta$ satisfy
\begin{equation}\label{alfabeta}
\frac{\alpha}{p}+\frac{\beta}{q}=1.
\end{equation}

The study of resonant systems has deserved a great deal of attention in recent years, and we may cite the works of Boccardo and de Figueiredo \cite{BF}, Manasevich and  Mawhin \cite{MM}, Felmer, Manasevich and de Th\`elin \cite{FMT}, Stavrakakis and Zographopoulos \cite{SZ2}, among several others.

In several applications, such as bifurcation problems, anti--maximum principles, and existence or non--existence of solutions (see for example \cite{ACl, DSZ, FMT, FGTdT, SZ2, VdT, Z}) it is desirable to have precise bounds on the eigenvalues. In general this information is not well understood for elliptic systems, except for the first or principal eigenvalue. Several properties of this first eigenvalue were analyzed (existence, uniqueness, positivity, and isolation in bounded or unbounded domains, with different boundary conditions and with or without weights) and we refer the interested reader to \cite{AH, dT, Fe, Z} among others.

Let us recall briefly that the existence of a sequence of variational eigenvalues for problem \eqref{uno}--\eqref{dos} was proved in \cite{dNPM}, and the values $\lam_k$ are defined as
$$
\lam_k := \inf _{C\in \mathcal{C}_k}\sup _{(u,v)\in C} \frac{\displaystyle \frac{1}{p}\int_{\Omega} |\nabla u|^{p}\, dx + \frac{1}{q} \int_{\Omega} |\nabla v|^{q}\, dx}{\displaystyle \int_{\Omega} |u|^{\alpha}|v|^{\beta}\, dx},
$$
where $\mathcal{C}_k$ is the class of compact symmetric ($C=-C$) subsets of $W^{1,p}_0(\Omega)\times W^{1,q}_0(\Omega)$ of (Krasnoselskii) genus greater or equal than $k$.

Throughout this work, the eigenvalues are counted repeated according to their multiplicity. We say that $\lambda_k$ has multiplicity $r$ if
$\lambda_{k-1}<\lambda_k=\lambda_{k+1}=\cdots=\lambda_{k+r-1}<\lambda_{k+r}$. In this case, it is a well know fact that the set of eigenfunctions corresponding to $\lambda_k$ has genus greater or equal than $r$ (see, for instance, \cite{G-A}).

In the case of a single equation the existence of a sequence of variational eigenvalues together with the correct order of growth for these eigenvalues was first obtained by Garcia-Azorero and Peral in \cite{G-A}. The constants in the asymptotic behavior were improved by Friedlander in \cite{Fr}. Let us note that for the one-dimensional problem, these bounds can be refined, see \cite{FBP1, FBP2}.

As for elliptic systems, the asymptotic growth of the eigenvalues is less understood even in the linear case. We may cite here the work of Protter \cite{Pr}, and also the works of Cantrell and Cosner \cite{Can, CC, Cos} were lower bounds for the first eigenvalue were obtained. The exception for the lack of results in this direction comes from linear and nonlinear elasticity theory (see the survey of Antman \cite{Ant}).

For nonlinear elliptic systems, up to our knowledge, the first work where this problem was addressed was \cite{dNP} where the authors obtained a generalization of the Lyapunov inequality together with an upper bound of the variational eigenvalues in terms of the ones of a single $p-$Laplace equation for the one dimensional case. Later, in \cite{FBP} we obtained lower and upper bounds for the eigenvalues of problem \eqref{uno}--\eqref{dos} in terms of the eigenvalues of a single $p-$Laplace and $q-$Laplace equations in any dimension $N \ge 1$. More precisely, for each $k$ we
prove that
$$
c \left(\frac{k}{|\Omega|}\right)^{\frac{q}{N}} \le \lam_k \le C \left(\frac{k}{|\Omega|}\right)^{\frac{p}{N}}
$$
for suitable positive constants $c, C$ depending on $p$ and $q$.

We refer the interested reader to the introductions of \cite{dNP, FBP} for more information and references about the eigenvalues of quasilinear elliptic equations and resonant systems.

However, our previous bounds fail to reflect the coupling strength of the system which is given by the parameters $\alpha$ and $\beta$. Formally, by taking $\alpha = p$ and $\beta = 0$, we obtain a single $p-$Laplace equation replacing the system, namely
$$
-\Delta_p u =  \lambda p |u|^{p -2}u,
$$
and similarly, for $\alpha=0$ and $\beta = q$ we have
$$
-\Delta_q v =  \lambda q|v|^{q-2}v.
$$
Hence, the order of growth of our upper (resp., lower) bound given in \cite{FBP} is sharp for the case $\alpha =p$ (resp., $\beta = q$), since coincides with the true upper (resp., lower) order of growth of the eigenvalues (see \cite{G-A}). On the other hand, both
orders does not hold simultaneously even for those limit cases.

We can suspect that there exists a smooth transition for the order of growth of the eigenvalues between both limiting cases, and the main result of this work is to prove it. Namely, our main theorem is
\begin{thm}\label{orden}
Let $\{\lam_k\}_{k\in\N}$ be the sequence of variational eigenvalues of problem \eqref{uno}--\eqref{dos}. Then, there exist positive constants $c < C$ depending on $p$, $q$, $\alpha$, $\beta$ and $N$, such that
$$
c \left(\frac{k}{|\Omega|}\right)^{\frac{\alpha+\beta}{N}} \le \lam_k \le C\left(\frac{k}{|\Omega|}\right)^{\frac{\alpha+\beta}{N}}
$$
when $k\to \infty$.
\end{thm}

Observe that if one consider linear operators, or even the same $p-$Laplace operator in both equations (i.e. $p=q$), the coupling parameters $\alpha$ and $\beta$ are not reflected in the asymptotics of the eigenvalues since $\alpha+\beta=p$ in this case. We believe that this fact may be the reason why this phenomenum was not discovered earlier.

The proof of Theorem \ref{orden} follows directly from  Weyl--type bounds for the spectral counting function $N(\lam)$ which gives the number of eigenvalues less than a given value, that is
$$
N(\lam) = \#\{ k : \lam_k \le \lam\}.
$$
Theorem \ref{orden} is equivalent to the following asymptotic bound for $N(\lam)$:
$$
C^{-1} \lam^{\frac{N}{\alpha+\beta}} \le N(\lam)  \le  c^{-1} \lam^{\frac{N}{\alpha+\beta}}.
$$

Up to our knowledge, this is the first case in the literature where the coupling parameters of an elliptic system appear explicitly modifying the power on the asymptotic order of growth of the eigenvalues.

\subsection{Organization of the paper}
The paper is organized as follows. In Section \S 2 we review some facts about the eigenvalue problem for the single $s-$Laplace equation in one space dimension
$$
-(|u^{\prime}|^{s-2}u^{\prime})^{\prime} =
|u|^{s-2}u,
$$
which can be found for example in \cite{DDM}, or \cite{DM}.

In Section \S 3 we prove Theorem \ref{orden} by using a scaling argument as in \cite{Fr}. The main drawback of this approach is the fact that the constant on the asymptotic expansion remains unknown since they depend on the first Dirichlet eigenvalue and the second Neumann eigenvalue of problem \eqref{uno}--\eqref{dos} when $\Omega$ is the unit cube $Q_1$.

So, we are left with the problem of finding lower and upper bounds for those eigenvalues, which is a problem of independent interest. Thus, we consider the one--dimensional problem
\begin{equation}\label{ode-sys}
\left\{
\begin{array}{rcl}
-(|u^{\prime}(x)|^{p-2}u^{\prime}(x))^{\prime} & = &
\lam\alpha |u|^{\alpha-2}u|v|^{\beta},\\
-(|v^{\prime}(x)|^{q-2}v^{\prime}(x))^{\prime} & = & \lam \beta
|u|^{\alpha}|v|^{\beta-2}v,
\end{array}\right.
\end{equation}
on the interval $(a,b)$, and we will focus on the Dirichlet boundary condition.

In Section \S 4 we prove the following lower bound for the first Dirichlet eigenvalue $\lam_1(p,q)$ of the one dimensional problem \eqref{ode-sys}:

\begin{thm}\label{lower}
Let $\Lambda_1(s)$ be the first Dirichlet eigenvalue of
$$
-(|\varphi^{\prime}|^{s-2}\varphi^{\prime})^{\prime} =
\Lambda |\varphi|^{s-2}\varphi
$$
on $(a,b)$. Then,
$$
\left(\alpha^{\frac{\alpha}{p}}\beta^{\frac{\beta}{q}}\right)^{-1} \left(\frac{2} {\pi_{\alpha+\beta}}\right)^{\alpha + \beta}  \frac{\Lambda_1(\alpha+\beta)}{\alpha+\beta-1} \le \lam_1(p,q).
$$
\end{thm}

The proof follows by using the Lyapunov inequality obtained in \cite{dNP}.

\medskip

In \cite{dNP} we obtained an upper bound of the first eigenvalue of the one dimensional problem \eqref{ode-sys} in terms of the first eigenvalue of the single $p-$Laplace equation. Moreover, upper bounds were obtained for all the variational eigenvalues, namely
\begin{equation}\label{cotita}
\lam_k\leq \frac{\Lambda_{k}(p)}{p}\Big[ 1  +  \left( \frac{p}{q}\right)^{q+1}
\big(\Lambda_{k}(p)\big)^{(q-p)/p}\Big].
\end{equation}
Here, $\Lambda_{k}(p)$ stands for the $k^{th}$ eigenvalue of the $p-$Laplace with Dirichlet boundary conditions. Let us note that \eqref{cotita} holds for the one dimensional problem. In the $N-$dimensional case, \eqref{cotita} holds only for the first eigenvalue.

In this paper we improve this explicit upper bound for the first eigenvalue, and we use it to obtain asymptotic bounds for the $k^{th}$ eigenvalue of a system in $\Omega\subset\R^N$ depending on the eigenvalues of the $(\alpha+\beta)-$Laplace operator.

\medskip

Section \S 5 is devoted to the proof of the following upper bound of the first Dirichlet eigenvalue of problem \eqref{ode-sys}:
\begin{thm}\label{upper}
Assume $q<p$. Then, with the notation of Theorem \ref{lower}, we have,
$$
\lam_1(p,q)\leq \Lambda_1(\alpha+\beta) \left[ \frac{1}{p} +\frac{1}{q} \left( \frac{\pi_{\alpha+\beta}}{(\alpha+\beta-1) \int_0^{\pi_{\alpha+\beta}} \sin_{\alpha+\beta}^{\alpha+\beta}(t)dt}  \right)^{1-\frac{q}{\alpha+\beta}}\right].
$$
\end{thm}

Here $\sin_s(x)$ is defined implicitely as
$$
x = \int_0^{\sin_s(x)} \frac{dt}{(1-t^s)^{1/s}}
$$
and $\pi_s$ is given by
$$
\pi_s = 2\int_0^1 \frac{dt}{(1-t^s)^{1/s}}.
$$
See section \S 2 for more details and see also the paper \cite{DDM}.

For the $N-$dimensional problem we prove the following
\begin{thm}\label{upper-N}
Let $\lambda_1(p,q)$ be the first eigenvalue of \eqref{uno} in the unit cube $Q_1\subset \R^N$. Assume again that $q<p$. Then we have the explicit upper bound
$$
\lambda_1(p,q)\le C\left( \frac{N(s-1)}{p}\pi_s^p + (s-1)^{q/s}\pi_s^q \Big(\int_0^1 |\phi_s(t)|^s\, dt\Big)^{-N(1-q/s)}\right),
$$
where $s=\alpha + \beta$, $C = \max\{1; N^{(p-2)/2}; N^{(q-2)/2}\}$ and $\phi_s(t)$ is the first eigenfunction of the one dimensional $s-$Laplace operaton with Dirichlet boundary conditions in the unit interval, that is
$$
\phi_s(t) = \sin_s(\pi_s t).
$$
\end{thm}

\medskip

Finally, in Section \S 6 we close this paper with some remarks on possible improvements on our result and some open questions.

\section{Some known facts}

In this Section we recall some previous results which will be needed in the rest of the paper.

\subsection{Variational setting}

The variational characterization of eigenvalues follows from the abstract theory developed by H. Amann (see \cite{A}). In \cite{DM} the authors showed the existence of infinitely many eigenpairs, i. e.
$(u,v) \in W^{1,p}_0(\Omega)\times W^{1,q}_0(\Omega)$ and $\lam\in \R$ such that
$$
\int_{\Omega} |\nabla u|^{p-2} \nabla u \nabla \varphi + |\nabla v|^{q-2} \nabla v \nabla \psi\, dx =
\lambda \int_{\Omega} ( \alpha |u|^{\alpha-2} u \phi |v|^\beta + \beta
|u|^\alpha |v|^{\beta-2} v \psi )\, dx
$$
for any test-function pair $(\phi,\psi)\in W^{1,p}_0(\Omega)\times W^{1,q}_0(\Omega)$.

It is convenient to work with the variational characterization of the eigenvalues, defined through the Rayleigh quotient,
\begin{equation}\label{varsist}
\lam_k=\inf _{C\in \mathcal{C}_k} \sup _{(u,v)\in C}
\frac{\displaystyle\frac{1}{p}\int_{\Omega}|\nabla u|^{p}\, dx +
\frac{1}{q}\int_{\Omega}|\nabla v|^{q}\, dx} {\displaystyle\int_{\Omega}|u|^{\alpha}|v|^{\beta}\, dx},
\end{equation}
where $\mathcal{C}_k$ is the class of compact symmetric ($C=-C$) subsets of $W^{1,p}_0(\Omega)\times W^{1,q}_0(\Omega)$ of (Krasnoselskii) genus greater or equal that $k$.

\subsection{One dimensional case}

For the one dimensional $s-$Laplace operator in $\Omega=(0,L)$
\begin{equation}\label{unaec}
-(|u'|^{s -2} u')' =  \Lambda  |u|^{s -2}u
\end{equation}
with Dirichlet boundary conditions, we have
\begin{equation}\label{formuvar}
\Lambda_{k}(s) = \inf _{C\in \mathcal{C}_k} \sup _{u\in C} \frac{\displaystyle\int_a^b |u'|^s\, dx}{\displaystyle \int_a^b |u|^s\, dx},
\end{equation}
with  $u\in W_0^{1,s}(0,L)$.

Here, all the eigenvalues and eigenfunctions can be found explicitly (see \cite{DM}):

\begin{thm}[Del Pino, Drabek and Manasevich, \cite{DDM}]\label{draman}
The eigenvalues $\Lambda_{k}(s)$ and eigenfunctions $u_{s,k}$ of equation \eqref{unaec} on the interval $[0,L]$ are given by
$$
\Lambda_{k}(s) = (s-1)\frac{\pi_s^s k^s}{L^s},
$$
$$
u_{s,k}(x) = \sin_s(\pi_s kx/L).
$$
\end{thm}

\begin{rem} It was proved in \cite{DM} that they coincide with the variational eigenvalues given by equation \eqref{formuvar}. However, let us observe that the notation is different in both papers.
\end{rem}

The function $\sin_s(x)$ is the solution of the initial value problem
$$
-(|u'|^{s -2} u')' =  (s-1) |u|^{s -2}u
$$
$$
u(0)=0, \qquad u'(0)=1,
$$
and is defined implicitly as
$$
x =  \int_0^{\sin_s(x)} \frac{dt}{(1-t^s)^{1/s}}.
$$
Moreover, its first zero is $\pi_s$, given by
$$
\pi_s = 2 \int_0^1 \frac{dt}{(1-t^s)^{1/s}}.
$$

Let us note that both $\sin_{s}(x)$ and $\sin_{s}'(x)$ satisfy
$$
|\sin_{s}(x)|\le 1, \quad |\sin_{s}'(x)|\le 1,
$$
due to the Pythagorean like identity
\begin{equation}\label{pyt}
|\sin_{s}(x)|^{s} + |\sin_{s}'(x)|^{s} = 1.
\end{equation}

Finally, let us observe that the following integral is a constant depending only on $s$:
$$
\int_0^{\pi_{s}} \sin_{s}^{s}(t)\, dt = K(s).
$$

\subsection{The Spectral Counting Function} Given the sequence $\{\lambda_k\}_{k\in \N}$, we introduce the spectral counting function $N(\lam)$  defined as
$$
N(\lam) =  \#\{k \colon \lam_k \le \lam \}.
$$

To avoid confusion, we will use $N_{sys}(\lam)$ or $N_s(\lam)$ to denote the eigenvalue counting functions of the system and the $s-$Laplace respectively. If necessary, we will write $N(\lam, \Omega)$ to denote explicitly the set $\Omega$ where
the eigenvalue problem is considered, and even $N^D(\lam)$ or $N^N(\lam)$ to indicate the Dirichlet and Neumann boundary conditions.

The main tool in order to obtain the asymptotic expansion of $N(\lam)$ is the classical Dirichlet-Neumann bracketing introduced by Courant \cite{CH} in a version due to \cite{FBP1}:
\begin{prop}[\cite{FBP1}, Theorem 2.1]
\label{dirneu}
Let $U_1, U_2 \in \R^N$ be disjoint open sets such that $(\overline{U_1 \cup U_2})^{int} = U $ and $|U \setminus U_1 \cup U_2|= 0$. Then,
\begin{align*}
N^D(\lam, U_1)+ N^D(\lam, U_2) &= N^D(\lam, U_1 \cup U_2 )\\
& \le N^D(\lam, U)\\
&\le  N^N(\lam, U)\\
& \le N^N(\lam, U_1 \cup U_2) \\
& =N^N(\lam, U_1)+ N^N(\lam, U_2).
\end{align*}
\end{prop}

\section{Estimates for the Spectral Counting Function}

Let us begin with the following scaling argument. We denote by $Q_1$ the unit cube in $\R^N$ and by $Q_t$ the scaled cube of side length $t$.
\begin{lema}\label{lemdir}
Let $\lam_1^1(p,q)$ (resp., $\lam_1^t(p,q)$) be the first Dirichlet eigenvalue of problem \ref{uno} when $\Omega = Q_1$ (resp.,  $\Omega = Q_t$). Then,
$$
t^{\alpha+\beta} \lam_1^t(p,q) = \lam_1^1(p,q).
$$
\end{lema}

\begin{proof}
First, observe that given any pair $(u,v)\in W^{1,p}_0(Q_1)\times W^{1,q}_0(Q_1)$, if we define
$$
u_t = tu\left(\frac{x}{t}\right),\qquad v_t = tv\left(\frac{x}{t}\right),
$$
then, $(u_t, v_t) \in  W^{1,p}_0(Q_t)\times W^{1,q}_0(Q_t)$. Also, it is clear that $(u,v)\mapsto (u_t, v_t)$ is a bijection between $W^{1,p}_0(Q_1)\times W^{1,q}_0(Q_1)$ and $W^{1,p}_0(Q_t)\times W^{1,q}_0(Q_t)$.

Now, by simple change of variables, we get
$$
\frac{\displaystyle\frac{1}{p}\int_{Q_t}|\nabla u_t|^{p}\, dx +
\frac{1}{q}\int_{Q_t}|\nabla v_t|^{q}\, dx}{\displaystyle\int_{Q_t}
|u_t|^{\alpha}|v_t|^{\beta}\, dx} = \frac{1}{t^{\alpha+\beta} } \displaystyle{ \frac{\displaystyle\frac{1}{p}\int_{Q_1}|\nabla u|^{p}\, dx +
\frac{1}{q}\int_{Q_1}|\nabla v|^{q}\, dx}{\displaystyle\int_{Q_1}
|u|^{\alpha}|v|^{\beta}\, dx} }.
$$

Finally, the variational characterization of the first eigenvalue gives the desired result, namely
$$
\lambda_1^t(p,q) = \frac{\lambda_1^1(p,q)}{t^{\alpha+\beta}}.
$$
This finishes the proof.
\end{proof}

\begin{rem} The argument used in the proof shows that if $(u,v)$ is an eigenfunction associated to $\lambda_1^1(p,q)$ then $(u_t,v_t)$ is an eigenfunction associated to $\lambda_1^t(p,q)$.
\end{rem}

By applying the very same argument we obtain the following relation between the second Neumann eigenvalues (recall that the first Neumann eigenvalue is $\mu_1^t(p,q)=0$).

\begin{lema}\label{lemneu}
Let $\mu_2^1(p,q)$ (resp., $\mu_2^t(p,q)$) be the second Neumann eigenvalue of problem \eqref{uno} when $\Omega = Q_1$ (resp., $\Omega = Q_t$). Then,
$$
\mu_2^t(p,q) = \frac{\mu_2^1(p,q)}{t^{\alpha+\beta} }.
$$
\end{lema}

\begin{proof}
The proof follows exactly as the one in Lemma \ref{lemdir} and the variational characterization of the second Neumann eigenvalue, namely

\begin{equation}
\mu_2^t(p,q) = \inf_{C_t\in \mathcal{C}_2} \sup_{(u_t,v_t)\in C_t}
\frac{\displaystyle\frac{1}{p}\int_{\Omega}|\nabla u_t|^{p}\, dx +
\frac{1}{q}\int_{\Omega}|\nabla v_t|^{q}\, dx}{\displaystyle\int_{\Omega}
|u_t|^{\alpha}|v_t|^{\beta}\, dx}.
\end{equation}
The proof is now completed.
\end{proof}

Now we are ready to prove Theorem \ref{orden}.

\begin{proof}[Proof of Theorem \ref{orden}]
For any $\lambda$ fixed, we take a lattice of cubes of side length $t\ll 1$ in $\R^N$ with $t$ depending on $\lam$.

First we derive a lower bound for $N(\lam)$ (equivalently, an upper  bound for $\lam_k$). From Lemma \ref{lemdir},
$\lam_1^t(p,q) = t^{-(\alpha+\beta)} \lam_1^1 (p,q)$, and taking
$$
t = \left(\frac{\lam}{\lam_1^1 (p,q)}\right)^{-\frac{1}{\alpha+\beta}},
$$
we have
$$
\lam_1^t(p,q) = \lambda
$$
Hence, since the first eigenvalue $\lambda_1^t(p,q)$ has multiplicity two (observe that $(u,v)$ is an eigenfunction if and only if $(u,-v)$ is an eigenfunction), we have $N^D(\lambda, Q_t)=2$. 

By using the Dirichlet Neumann bracketing, a lower bound for $N(\lam)$ is given by $2M$, where $M$ is the number of cubes of the lattice contained in $\Omega$. Since
$$
t^N M \to |\Omega|
$$
when $t\to 0$, we have
$$
N(\lam) \ge  2M \sim 2\frac{|\Omega|}{t^N} = 2|\Omega| \left(\frac{\lam}{\lam_1^1 (p,q)}\right)^{\frac{N}{\alpha+\beta}}.
$$

The upper bound for $\lam_k(p,q)$ follows since
$$
k = N(\lam_k(p,q)) \ge \frac{2|\Omega|}{\lam_1^1 (p,q)^{\frac{N}{\alpha+\beta}}} \, \lam_k(p,q)^{\frac{N}{\alpha+\beta}}.
$$

Let us find an upper bound for $N(\lam)$. We use the bound for the second Neumann eigenvalue proved in Lemma \ref{lemneu},
$\mu_2^t(p,q) = t^{-(\alpha+\beta)} \mu_2^1(p,q)$. Hence, taking
$$
t = \left(\frac{c\lam}{\mu_2^1(p,q)}\right)^{-\frac{1}{\alpha+\beta}},
$$
for any $c>1$, we have
$$
\mu_2^t(p,q) = c\lambda > \lambda.
$$
Therefore, for each cube we have $N^N(\lambda, Q_t)=1$. 

By using again the Dirichlet Neumann bracketing, an upper bound for $N(\lam)$ is given by $M$, where $M$ is the number of cubes covering $\Omega$.  As before, we have
$$
N(\lam) \le M \sim \frac{|\Omega|}{t^N} = |\Omega| \left(\frac{c \lam}{ \mu_2^1 (p,q)}\right)^{\frac{N}{\alpha+\beta}}.
$$

The lower bound for $\lam_k(p,q)$ follows  since
$$
k = N(\lam_k(p,q)) \le \frac{|\Omega|}{\mu_2^1 (p,q)^{\frac{N}{\alpha+\beta}}} \, (c \lam_k(p,q))^{\frac{N}{\alpha+\beta}}.
$$

The Theorem is proved.\end{proof}

\section{A Lower Bound for the First Eigenvalue}

In this section we restrict ourselves to the one dimensional case and prove Theorem \ref{lower}. For that purpose, we use the following Lyapunov inequality for
systems proved in \cite{dNP}:
\begin{thm}[\cite{dNP}, Theorem 1.5]\label{lyap}
Let us assume that there exists a positive solution of the system
$$
\left\{ \begin{array}{rcl}
-(|u^{\prime}(x)|^{p-2}u^{\prime}(x))^{\prime} & = & f(x)|u|^{\alpha-2}u|v|^{\beta}\\
-(|v^{\prime}(x)|^{q-2}v^{\prime}(x))^{\prime} & = &
g(x)|u|^{\alpha}|v|^{\beta-2}v\end{array}\right.
$$
on the interval $(a,b)$, with Dirichlet boundary conditions. Then, we have that:
\begin{equation}\label{lyapsys}
2^{\alpha+\beta}\leq(b-a)^{\frac{\alpha}{p^{\prime}}+\frac{\beta}{q^{\prime}}}
\left(\int_{a}^{b}f(x)\, dx\right)^{\frac{\alpha}{p}} \left(\int_{a}^{b}g(x)\,
dx\right)^{\frac{\beta}{q}}
\end{equation}
\end{thm}

This result gives Theorem \ref{lower} after replacing $f(x)= \alpha \lam_1(p,q)$ and $g(x)=\beta \lam_1(p,q)$. In fact, we have
\begin{equation}\label{lyapsyseig}
2^{\alpha+\beta}\leq  \alpha^{\frac{\alpha}{p}}\beta^{\frac{\beta}{q}}
(b-a)^{\frac{\alpha}{p^{\prime}}+\frac{\beta}{q^{\prime}}+
\frac{\alpha}{p}+\frac{\beta}{q}} \lam_1(p,q),
\end{equation}
and let us note that
$$
\frac{\alpha}{p^{\prime}}+\frac{\beta}{q^{\prime}}+
\frac{\alpha}{p}+\frac{\beta}{q} = \alpha\left(
\frac{1}{p}+\frac{1}{p^{\prime}}\right) + \beta\left( \frac{1}{q}+
\frac{1}{q^{\prime}}\right) = \alpha + \beta.
$$
So,
$$
2^{\alpha+\beta}\leq \alpha^{\frac{\alpha}{p}} \beta^{\frac{\beta}{q}} (b-a)^{\alpha + \beta} \lam_1(p,q).
$$

The desired result follows from this inequality and the explicit formula for $\Lambda_1(\alpha+\beta)$ in Theorem \ref{draman}, since:
$$
\frac{2^{\alpha+\beta}}{(b-a)^{\alpha + \beta}} \frac{\alpha+\beta-1}{\alpha+\beta-1}
\left(\frac{\pi_{\alpha+\beta}}{ \pi_{\alpha+\beta}}\right)^{\alpha + \beta}=
\frac{1}{\alpha+\beta-1}  \left(\frac{2} {\pi_{\alpha+\beta}}\right)^{\alpha + \beta}
\Lambda_1(\alpha+\beta),
$$
and therefore
$$
\frac{1}{\alpha+\beta-1}  \left(\frac{2} {\pi_{\alpha+\beta}}\right)^{\alpha + \beta}
\Lambda_1(\alpha+\beta) \le
\alpha^{\frac{\alpha}{p}}\beta^{\frac{\beta}{q}}\lam_1(p,q).
$$

This completes the proof of Theorem \ref{lower}.\qed

\section{An Upper Bound for the First Eigenvalue}

In this section, we first consider the one dimensional case and prove Theorem \ref{upper}.
Then, we extend the bound to the $n-$dimensional case by means of comparing the eigenvalue of the $p-$Laplacian with that of the pseudo $p-$Laplacian.

For the upper bound of the first eigenvalue, we need to improve the bound given in \cite{dNP}.

Given the variational characterization of the first eigenvalue,
$$
\lambda_1(p,q) \leq \frac{\displaystyle \frac{1}{p}\int_0^L
|u'|^p\, dx + \frac{1}{q}\int_0^L | v'|^q\, dx} {\displaystyle \int_0^L |u|^\alpha
|v|^{\beta}\, dx}.
$$
We denote $s=\alpha + \beta$ and choose $u=v=\varphi_1$,  which is a multiple of the first Dirichlet eigenfunction of the single equation
$$
-(|w'|^{s-2} w')'= \Lambda  |w|^{s-2}w,
$$
with Dirichlet boundary conditions, that is
$$
\varphi_1(x) = \frac{L}{\pi_s} \sin_s(\pi_s x/L).
$$
Due to the Pythagorean--like identity \eqref{pyt} we have
$$
|\varphi_1'(x)|\le 1.
$$

So, taking $u=v=\varphi_1$ in the Rayleigh quotient we get
$$
\lambda_1(p,q) \leq \frac{\displaystyle \frac{1}{p}\int_0^L |\varphi_1'|^p\, dx + \frac{1}{q}\int_0^L |\varphi_1'|^q\, dx} {\displaystyle \int_0^L |\varphi_1|^s\, dx}
$$

Now, as $q<p$ we have that $s = \alpha+\beta<p$. Moreover, as $|\varphi_1'(x)| \le 1$, the inequality
$$
|\varphi_1'(x)|^p \le |\varphi_1'(x)|^s
$$
holds. On the other hand, by H\"older's inequality we obtain
$$
\int_0^L | \varphi_1'|^q\, dx \le \left(\int_0^L | \varphi_1'|^s\, dx\right)^{\frac{q}{s}}
L^{1-\frac{q}{s}}
$$

Collecting all these facts, we have
$$
\begin{array}{rl}
\lambda_1(p,q) & \leq \displaystyle \frac{1}{p} \frac{\displaystyle \int_0^L |\varphi_1'|^s\, dx} {\displaystyle \int_0^L |\varphi_1|^s\, dx} + \frac{1}{q} \frac{\displaystyle \left(\int_0^L |\varphi_1'|^s\, dx \right)^{\frac{q}{s}} L^{1-\frac{q}{s}}} {\displaystyle \int_0^L |\varphi_1|^s\, dx} \\ 
\\ 
&\displaystyle \leq \frac{1}{p}
\Lambda_1(s) + \frac{1}{q} \Lambda_1(s)^{\frac{q}{s}} \frac{L^{1-\frac{q}{s}}}{\displaystyle \left(\int_0^L |
\varphi_1|^s\, dx\right)^{1-\frac{q}{s}}}
\end{array}
$$

Now we need an upper bound for
$$
\left( \frac{L}{\displaystyle \int_0^L | \varphi_1|^s \, dx}\right)^{1-\frac{q}{s}}.
$$
Indeed, we can compute the integral explicitly, obtaining
$$
\begin{array}{rl}
\displaystyle{\int_0^L | \varphi_1|^s\, dx} & = \displaystyle{
\left(\frac{L}{\pi_s}\right)^s \int_0^L \sin_s^s(\pi_s x/L)dx} \\
\\ & \displaystyle{=  \left(\frac{L}{\pi_s}\right)^{s+1} \int_0^{\pi_s}
\sin_s^s(t)\, dt }\\ \\
& \displaystyle{=
\left(\frac{L}{\pi_s}\right)^{s+1} K(s).}
\end{array}
$$
Then,
$$
\lambda_1(p,q) \leq \frac{1}{p} \Lambda_1(s) + \frac{1}{q} \Lambda_1(s)^{\frac{q}{s}}
\left( \frac{\pi_s^{s+1}} {L^s K(s)} \right)^{1-\frac{q}{s}}.
$$

Now, by using the explicit formula for $\Lambda_1(s)$, we have
$$
\lambda_1(p,q) \leq \Lambda_1(s) \left[ \frac{1}{p}
+\frac{1}{q} \left( \frac{\pi_s} {(s-1)K(s)} \right)^{1-\frac{q}{s}}\right].
$$

The proof of Theorem \ref{upper} is finished.\qed

\medskip

Now we go back to the $N-$dimensional problem and try to find an explicit upper bound for the first eigenvalue $\lambda_1(p,q)$ in the unit cube $Q_1$. That is, we prove Theorem \ref{upper-N}.

In order to do this, we recall the pseudo $p-$Laplace operator
$$
\hat \Delta_p u = \sum_{i=i}^N (|u_{x_i}|^{p-2}u_{x_i})_{x_i}
$$
and consider the eigenvalue problem associated to $\hat \Delta_p$
$$
\nu_1 = \inf_{(u,v)\in W^{1,p}_0(Q_1)\times W^{1,q}_0(Q_1)} \frac{\displaystyle \int_{Q_1} \frac{1}{p} |\nabla u|_p^p + \frac{1}{q} |\nabla v|_q^q\, dx}{\displaystyle \int_{Q_1} |u|^\alpha |v|^\beta\, dx}
$$
where $|x|_s^s = \sum_{i=1}^N |x_i|^s$.

By the equivalence of norms
$$
\begin{array}{ll}
|x|_2 \le N^{(s-2)/2} |x|_s & \qquad \mbox{if } p>2,\\
|x|_2 \le |x|_s & \qquad \mbox{if } 1<p\le 2, 
\end{array}
$$
we have that $\lambda_1(p,q)\le C \nu_1$, where $C=\max\{1; N^{(p-2)/2}; N^{(q-2)/2}\}$. So we need to bound $\nu_1$.

Recall that the first eigenfunction of
$$
-\hat\Delta_s u = \nu_1 |u|^{s-2}u\qquad  \mbox{in } Q_1
$$
with Dirichlet boundary conditions can be computed explicitly by separation of variables.

In fact, let $\phi_s(x)=\sin_s(\pi_s x)$ be the first eigenfunction of the one dimensional $s-$Laplace operator with Dirichlet boundary conditions in the interval $(0,1)$, then
$$
w_s(x) = \prod_{i=1}^n \phi_s(x_i)
$$
is the first eigenfunction of the $N-$dimensional pseudo $s-$laplacian in the cube $Q_1$.

Now, as in the one dimensional case, we use $(w_s,w_s)$, with $s=\alpha+\beta$ as a test function for $\nu_1$ to get
\begin{align*}
\nu_1 \le& \frac{\displaystyle \frac{N}{p} \Big(\int_0^1 |\phi_s(t)|^p\, dt\Big)^{N-1} \int_0^1 |\phi_s'(t)|^p\, dt}{\displaystyle \Big(\int_0^1 |\phi_s(t)|^s\, dt\Big)^N} + \frac{\displaystyle \frac{N}{q} \Big(\int_0^1 |\phi_s(t)|^q\, dt\Big)^{N-1} \int_0^1 |\phi_s'(t)|^q\, dt}{\displaystyle \Big(\int_0^1 |\phi_s(t)|^s\, dt\Big)^N}.
\end{align*}

By the Pytagorian--like identity
$$
|\sin_s(t)|^s + |\sin_s'(t)|^s = 1
$$
we conclude that $|\sin_s(t)|, |\sin_s'(t)|\le 1$.

Also, as $q<s<p$ we get $|\phi_s(t)|^p\le |\phi_s(t)|^s$ and $|\phi_s'(t)|^p\le \pi_s^{p-s} |\phi_s'(t)|^s$

So
\begin{align*}
\nu_1&\le \frac{N}{p}\pi_s^{p-s}(s-1)\pi_s^s +  \frac{\displaystyle \frac{N}{q} \Big(\int_0^1 |\phi_s(t)|^q\, dt\Big)^{N-1} \int_0^1 |\phi_s'(t)|^q\, dt} {\displaystyle \Big(\int_0^1 |\phi_s(t)|^s\, dt\Big)^N}\\
&= \frac{N(s-1)}{p}\pi_s^p +  \frac{\displaystyle \frac{N}{q} \Big(\int_0^1 |\phi_s(t)|^q\, dt\Big)^{N-1} \int_0^1 |\phi_s'(t)|^q\, dt} {\displaystyle \Big(\int_0^1 |\phi_s(t)|^s\, dt\Big)^N}.
\end{align*}

To bound the other term, we use H\"older's inequality to get
\begin{align*}
\Big(\int_0^1 |\phi_s(t)|^q\, dt\Big)^{N-1} \le& \Big(\int_0^1 |\phi_s(t)|^s\, dt\Big)^{(N-1)q/s},\\
\int_0^1 |\phi_s'(t)|^q\, dt \le& \Big(\int_0^1|\phi_s'(t)|^s\, dt\Big)^{q/s} = \Big((s-1)\pi_s^s \int_0^1 |\phi_s(t)|^s\, dt\Big)^{q/s}.
\end{align*}

So
$$
\nu_1 \le \frac{N(s-1)}{p}\pi_s^p + (s-1)^{q/s}\pi_s^q \Big(\int_0^1 |\phi_s(t)|^s\, dt\Big)^{-N(1-q/s)}
$$
which is an explicit bound for $\nu_1$.
\qed

\section{An explicit lower bound on the Main Theorem}

In Section \S 3 we obtained the following asymtotic bounds for the spectral counting function:
$$
\frac{2|\Omega|}{\lam_1^1 (p,q)^{\frac{N}{\alpha+\beta}}} \, \lam^{\frac{N}{\alpha+\beta}} \le N(\lam)
\le  \frac{|\Omega|}{\mu_2^1 (p,q)^{\frac{N}{\alpha+\beta}}} \, (c \lam)^{\frac{N}{\alpha+\beta}}.
$$

An explicit lower bound was given in the previous section by means of the explicit upper bound for $\lambda_1^1(p,q)$.

Although the previous formula holds for any constant $c>1$, it is convenient to take $c= 1+ \lam^{-1}$, since in this case we can rewrite the upper bound as
$$
\frac{2|\Omega|}{\mu_2^1 (p,q)^{\frac{N}{\alpha+\beta}}} \, \lam^{\frac{N}{\alpha+\beta}} + O(1).
$$

We conjecture that a stronger result holds, namely,
$$
N(\lam) = c(\Omega, \alpha,\beta)\lam^{\frac{N}{\alpha+\beta}} + o(\lam^{\frac{N}{\alpha+\beta}}).
$$
Indeed, the proof follows immediately if it is true that
$$
\lam_1^1 (p,q) = 2\mu_2^1 (p,q).
$$
For a single equation, the factor 2 does not enter since the first eigenvalue is simple. In this case, the equality $\lambda_1=\mu_2$ holds for a single equation in dimension one. Up to our knowledge, it is not known for systems, nor in the case $N>1$ even for a single $p-$Laplace equation.

\section*{Acknowledgements}  Supported by grant X078 from Universidad de Buenos Aires, grants 06-835, 06-290 from ANPCyT PICT. Both authors are members of CONICET (Argentina).


\begin{thebibliography}{10}

\bibitem{AH}
Walter Allegretto and Yin~Xi Huang.
\newblock A {P}icone's identity for the {$p$}-{L}aplacian and applications.
\newblock {\em Nonlinear Anal.}, 32(7):819--830, 1998.

\bibitem{A}
Herbert Amann.
\newblock Lusternik-{S}chnirelman theory and non-linear eigenvalue problems.
\newblock {\em Math. Ann.}, 199:55--72, 1972.

\bibitem{Ant}
Stuart~S. Antman.
\newblock The influence of elasticity on analysis: modern developments.
\newblock {\em Bull. Amer. Math. Soc. (N.S.)}, 9(3):267--291, 1983.

\bibitem{ACl}
C{\'e}line Azizieh and Philippe Cl{\'e}ment.
\newblock A priori estimates and continuation methods for positive solutions of
  {$p$}-{L}aplace equations.
\newblock {\em J. Differential Equations}, 179(1):213--245, 2002.

\bibitem{BF}
Lucio Boccardo and Djairo Guedes~de Figueiredo.
\newblock Some remarks on a system of quasilinear elliptic equations.
\newblock {\em NoDEA Nonlinear Differential Equations Appl.}, 9(3):309--323,
  2002.

\bibitem{Can}
Robert~Stephen Cantrell.
\newblock On coupled multiparameter nonlinear elliptic systems.
\newblock {\em Trans. Amer. Math. Soc.}, 294(1):263--285, 1986.

\bibitem{CC}
Robert~Stephen Cantrell and Chris Cosner.
\newblock On the generalized spectrum for second-order elliptic systems.
\newblock {\em Trans. Amer. Math. Soc.}, 303(1):345--363, 1987.

\bibitem{Cos}
Chris Cosner.
\newblock Estimates for eigenfunctions and eigenvalues of nonlinear elliptic
  problems.
\newblock {\em Trans. Amer. Math. Soc.}, 282(1):59--75, 1984.

\bibitem{CH}
R.~Courant and D.~Hilbert.
\newblock {\em Methods of mathematical physics. {V}ol. {I}}.
\newblock Interscience Publishers, Inc., New York, N.Y., 1953.

\bibitem{dNPM}
Pablo~L. de~N{\'a}poli and M.~Cristina Marianni.
\newblock Quasilinear elliptic systems of resonant type and nonlinear
  eigenvalue problems.
\newblock {\em Abstr. Appl. Anal.}, 7(3):155--167, 2002.

\bibitem{dNP}
Pablo~L. De~N{\'a}poli and Juan~P. Pinasco.
\newblock Estimates for eigenvalues of quasilinear elliptic systems.
\newblock {\em J. Differential Equations}, 227(1):102--115, 2006.

\bibitem{dT}
F.~de~Th{\'e}lin.
\newblock Premi\`ere valeur propre d'un syst\`eme elliptique non lin\'eaire.
\newblock {\em Rev. Mat. Apl.}, 13(1):1--8, 1992.

\bibitem{DDM}
Manuel del Pino, Pavel Dr{\'a}bek, and Raul Man{\'a}sevich.
\newblock The {F}redholm alternative at the first eigenvalue for the
  one-dimensional {$p$}-{L}aplacian.
\newblock {\em J. Differential Equations}, 151(2):386--419, 1999.

\bibitem{DSZ}
P.~Dr{\'a}bek, N.~M. Stavrakakis, and N.~B. Zographopoulos.
\newblock Multiple nonsemitrivial solutions for quasilinear elliptic systems.
\newblock {\em Differential Integral Equations}, 16(12):1519--1531, 2003.

\bibitem{DM}
Pavel Dr{\'a}bek and Ra{\'u}l Man{\'a}sevich.
\newblock On the closed solution to some nonhomogeneous eigenvalue problems
  with {$p$}-{L}aplacian.
\newblock {\em Differential Integral Equations}, 12(6):773--788, 1999.

\bibitem{FMT}
Patricio Felmer, Ra{\'u}l~F. Man{\'a}sevich, and Fran{\c{c}}ois de~Th{\'e}lin.
\newblock Existence and uniqueness of positive solutions for certain
  quasilinear elliptic systems.
\newblock {\em Comm. Partial Differential Equations}, 17(11-12):2013--2029,
  1992.

\bibitem{Fe}
Juli{\'a}n Fern{\'a}ndez~Bonder.
\newblock Multiple positive solutions for quasilinear elliptic problems with
  sign-changing nonlinearities.
\newblock {\em Abstr. Appl. Anal.}, (12):1047--1055, 2004.

\bibitem{FBP}
Juli{\'a}n Fern{\'a}ndez~Bonder and Juan~P. Pinasco.
\newblock Estimates for eigenvalues of quasilinear elliptic systems. {II}.
\newblock {\em J. Differential Equations}, 245(4):875--891, 2008.

\bibitem{FBP1}
Juli{\'a}n Fern{\'a}ndez~Bonder and Juan~Pablo Pinasco.
\newblock Asymptotic behavior of the eigenvalues of the one-dimensional
  weighted {$p$}-{L}aplace operator.
\newblock {\em Ark. Mat.}, 41(2):267--280, 2003.

\bibitem{FBP2}
Juli{\'a}n Fern{\'a}ndez~Bonder and Juan~Pablo Pinasco.
\newblock Eigenvalues of the {$p$}-{L}aplacian in fractal strings with
  indefinite weights.
\newblock {\em J. Math. Anal. Appl.}, 308(2):764--774, 2005.

\bibitem{FGTdT}
Jacqueline Fleckinger-Pell{\'e}, Jean-Pierre Gossez, Peter Tak{\'a}{\v{c}}, and
  Fran{\c{c}}ois de~Th{\'e}lin.
\newblock Nonexistence of solutions and an anti-maximum principle for
  cooperative systems with the {$p$}-{L}aplacian.
\newblock {\em Math. Nachr.}, 194:49--78, 1998.

\bibitem{Fr}
Leonid Friedlander.
\newblock Asymptotic behavior of the eigenvalues of the {$p$}-{L}aplacian.
\newblock {\em Comm. Partial Differential Equations}, 14(8-9):1059--1069, 1989.

\bibitem{G-A}
Jes{\'u}s Garc{\'{\i}}a~Azorero and Ireneo Peral~Alonso.
\newblock Comportement asymptotique des valeurs propres du {$p$}-laplacien.
\newblock {\em C. R. Acad. Sci. Paris S\'er. I Math.}, 307(2):75--78, 1988.

\bibitem{MM}
Ra{\'u}l Man{\'a}sevich and Jean Mawhin.
\newblock The spectrum of {$p$}-{L}aplacian systems with various boundary
  conditions and applications.
\newblock {\em Adv. Differential Equations}, 5(10-12):1289--1318, 2000.

\bibitem{Pr}
M.~H. Protter.
\newblock The generalized spectrum of second-order elliptic systems.
\newblock {\em Rocky Mountain J. Math.}, 9(3):503--518, 1979.

\bibitem{SZ2}
N.~M. Stavrakakis and N.~B. Zographopoulos.
\newblock Bifurcation results for quasilinear elliptic systems.
\newblock {\em Adv. Differential Equations}, 8(3):315--336, 2003.

\bibitem{VdT}
Jean V{\'e}lin and Fran{\c{c}}ois de~Th{\'e}lin.
\newblock Existence and nonexistence of nontrivial solutions for some nonlinear
  elliptic systems.
\newblock {\em Rev. Mat. Univ. Complut. Madrid}, 6(1):153--194, 1993.

\bibitem{Z}
N.~B. Zographopoulos.
\newblock {$p$}-{L}aplacian systems on resonance.
\newblock {\em Appl. Anal.}, 83(5):509--519, 2004.

\end{thebibliography}
\end{document}